\documentclass[a4paper, article, 11pt]{jloganal}

\usepackage{pinlabel}
\usepackage{amsmath}
\usepackage{amssymb,latexsym}
\usepackage{amsfonts}
\usepackage{amsthm}
\usepackage{graphicx}
\usepackage{lineno}

\usepackage{enumerate}

\newtheorem{theorem}{Theorem}[section]

\theoremstyle{definition}

 \newtheorem{lemm}{Lemma}
 
 \newtheorem{propi}{Proposition}
 \theoremstyle{definition}
 
 \newtheorem{satz}{Theorem}
 \newtheorem{question}{Question}
 
 \theoremstyle{remark}

 \newtheorem*{clm}{Claim}

\theoremstyle{remark}
\newtheorem{remark}[theorem]{Remark}

 \newcommand{\la}{\langle}
 \newcommand{\ra}{\rangle}
 \newcommand{\rest}[1]{\upharpoonright{#1}}
 
 \newcommand{\dom}{\hbox{dom}}

 \newcommand{\M}{\mathcal{M}}
 \newcommand{\N}{\mathcal{N}}
 \newcommand{\Ca}{\mathcal{C}}

 \newcommand{\PP}{\mathbb{P}}
 
 \newcommand{\QQ}{\mathbb{Q}}

 \newcommand{\RR}{\mathbb{R}}
 \newcommand{\KK}{\mathbb{K}}

  \newcommand{\LL}{\mathcal{L}} 

\newcommand{\w}{\omega}

\DeclareMathOperator{\lh}{lh}
\DeclareMathOperator{\cof}{cof}

\title[Projective maximal families of orthogonal measures with large continuum]{Projective maximal families of orthogonal measures\\ with large continuum}

\author[V Fischer]{Vera Fischer} 
\givenname{Vera} 
\surname{Fischer} 
\address{Kurt G\"odel Research Center, University of Vienna, W\"ahringer Strasse 25, A-1090 Vienna, Austria} 
\email{vfischer@logic.univie.ac.at}

\author[S D Friedman]{Sy David Friedman} 
\givenname{Sy David} 
\surname{Friedman} 
\address{Kurt G\"odel Research Center, University of Vienna, W\"ahringer Strasse 25, A-1090 Vienna, Austria} 
\email{sdf@logic.univie.ac.at}

\author[A Tornquist]{Asger T\"ornquist} 
\givenname{Asger} 
\surname{T\"ornquist} 
\address{Kurt G\"odel Research Center, University of Vienna, W\"ahringer Strasse 25, A-1090 Vienna, Austria} 
\email{asger@logic.univie.ac.at}


\begin{document}

\begin{abstract}
We study maximal orthogonal families of Borel probability measures on $2^\omega$ (abbreviated m.o. families) and show that there
are generic extensions of the constructible universe $L$ in which each of the following holds:

(1) There is a $\Delta^1_3$-definable well order of the reals, there is a $\Pi^1_2$-definable m.o. family, there are no $\mathbf{\Sigma}^1_2$-definable
m.o. families and $\mathfrak{b}=\mathfrak{c}=\omega_3$ (in fact any reasonable value of $\mathfrak{c}$ will do).

(2) There is a $\Delta^1_3$-definable well order of the reals, there is a $\Pi^1_2$-definable m.o. family, there are no $\mathbf{\Sigma}^1_2$-definable
m.o. families, $\mathfrak{b}=\omega_1$ and $\mathfrak{c}=\omega_2$.

\end{abstract}

\maketitle

\section{Introduction}

Let $X$ be a Polish space, and let $P(X)$ denote the Polish space of Borel probability measures on $X$, in the sense of \cite[17.E]{KECHRIS}. Recall that if $\mu,\nu\in P(X)$ then $\mu$ and $\nu$ are said to be \emph{orthogonal}, written $\mu\bot\nu$, if there is a Borel set $B\subseteq X$ such that $\mu(B)=0$ and $\nu(X\setminus B)=0$. A set of measures $\mathcal A\subseteq P(X)$ is said to be \emph{orthogonal} if whenever $\mu,\nu\in\mathcal A$ and $\mu\neq\nu$ then $\mu\perp\nu$. A \emph{maximal orthogonal family}, or \emph{m.o.~family}, is an orthogonal family $\mathcal A\subseteq P(X)$ which is maximal under inclusion.

The present paper is concerned with the study of \emph{definable} m.o.~families. A well-known result to Preiss and Rataj \cite{PR} states that there are no analytic m.o.~families, and in a recent paper \cite{VFAT} it was shown by Fischer and T\"ornquist that if all reals are constructible then there is a $\Pi^1_1$ m.o.~family. The latter paper also raised the question how restrictive the existence of a definable m.o.~family is on the structure of the real line, since it was shown that $\Pi^1_1$ m.o. families cannot coexist with Cohen reals.

In the present paper we study $\Pi^1_2$ m.o.~families in the context of $\mathfrak c\geq\omega_2$, with the additional requirement that there is a $\Delta^1_3$-definable wellorder of $\RR$. Our main results are:

\begin{satz}\label{TH1} It is consistent with $\mathfrak{c}=\mathfrak{b}=\omega_3$ that there is a $\Delta^1_3$-definable wellorder
of the reals, a $\Pi^1_2$ definable maximal orthogonal family of measures and there are no $\mathbf{\Sigma}^1_2$-definable maximal sets of orthogonal measures.
\end{satz}

There is nothing special about $\mathfrak{c}=\omega_3$. In fact the same result can be obtained for any reasonable value of $\mathfrak{c}$.

\begin{satz}\label{TH2} It is consistent with $\mathfrak{b}=\omega_1$, $\mathfrak{c}=\omega_2$ that there is a
$\Delta^1_3$-definable wellorder of the reals, a $\Pi^1_2$
definable maximal orthogonal family of measures and there are no $\mathbf{\Sigma}^1_2$-definable maximal sets of orthogonal measures.
\end{satz}

Taken together these theorems indicate that the existence of a $\Pi^1_2$ m.o.~family does not seem to impose any severe restrictions on the structure of the real line. On the other hand, we show (Proposition \ref{PR1}) that $\Sigma^1_2$ m.o.~families cannot coexist with either Cohen or random reals, which is why in the models produced to prove Theorems \ref{TH1} and \ref{TH2} there are no $\mathbf{\Sigma}^1_2$ m.o.~families.

The theorems of this paper belong to a line of results concerning the definability of certain combinatorial objects on the real line and in particular the question of how low in the projective hierarchy such objects exist. In~\cite{ADM} Mathias showed that there is no $\Sigma^1_1$-definable maximal almost disjoint (mad) family in $[\omega]^\omega$. Assuming $V=L$, Miller obtained (see~\cite{MILL}) a $\Pi^1_1$ mad family in $[\omega]^\omega$.

The study of the existence of definable combinatorial objects on $\RR$ in the presence of a projective wellorder of the reals and $\mathfrak{c}\geq\omega_2$ was initiated in~\cite{VFSF},~\cite{SFLZ} and~\cite{VFSFLZ}. The wellorder of $\RR$ in all those models has a $\Delta^1_3$-definition, which is indeed optimal for models of $\mathfrak{c}\geq\omega_2$, since by Mansfield's  theorem (see~\cite[Theorem 25.39]{MSF}) the existence of a $\Sigma^1_2$-definable wellorder of the reals implies that all reals are constructible. The existence of a $\Pi^1_2$-definable $\omega$-mad family in $[\omega]^\omega$ in the presence of $\mathfrak{c}=\mathfrak{b}=\omega_2$ was established by Friedman and Zdomskyy in~\cite{SFLZ}. In the same paper, referring to earlier results (see~\cite{RAG} and~\cite{KASJSYZ}) they outlined the construction of a model in which $\mathfrak{c}=\omega_2$ and there is a $\Pi^1_1$-definable $\omega$-mad family: Start with the constructible universe $L$, obtain a $\Pi^1_1$-definable $\omega$ mad family and proceed with a countable support iteration of length $\omega_2$ of Miller forcing.
The techniques were further developed in~\cite{VFSFLZ} to establish a model in which there is a $\Pi^1_2$-definable $\omega$-mad family and $\mathfrak{c}=\mathfrak{b}=\omega_3$. In particular, in the models from~\cite{SFLZ} and~\cite{VFSFLZ}, there are no maximal almost disjoint families of size  $<\mathfrak{c}$ and so the almost disjointness number has a $\Pi^1_2$-witness.

The present paper combines the encoding techniques of \cite{VFAT} with the techniques of \cite{VFSF, SFLZ,VFSFLZ} to obtain Theorems \ref{TH1} and \ref{TH2}. We note that one significant difference from the situation for mad families is that m.o.~families always have size $\mathfrak{c}$ (see~\cite[Proposition 4.1]{VFAT}).

{\it Acknowledgement.} The authors would like to thank the Austrian Science Fund FWF for the generous support through
grants no. P 20835-N13 (Fischer, Friedman), and P 19375-N18 (Friedman, T\"ornquist), as well as a Marie Curie grant from the European Union no. IRG-249167 (T\"ornquist).


\section{Preliminaries}

In this section, we briefly recall the coding of probability measures on $2^\omega$ and the encoding technique for measures introduced in \cite{VFAT}.

Let $X$ be a Polish space. Recall that measures if $\mu,\nu\in P(X)$ then $\mu$ is said to be {\emph{absolutely continuous}} with respect to $\nu$, written $\mu\ll\nu$, if for all Borel subsets of $X$ we have that $\nu(B)=0$ implies that $\mu(B)=0$. Two measures $\mu,\nu\in P(2^\omega)$ are called {\emph{absolutely equivalent}}, written $\mu\approx\nu$, if $\mu\ll\nu$ and $\nu\ll\mu$.

If $s\in 2^{<\omega}$ we let $N_s=\{x\in 2^\omega:s\subseteq x\}$ be the basic neighbourhood determined by $s$. Following \cite{VFAT},  we let
$$
p(2^\omega)=\{f:2^{<\omega}\to [0,1]:f(\emptyset)=1\wedge (\forall s\in 2^{<\omega}) f(s)=f(s^\smallfrown 0)+f(s^\smallfrown 1)\}.
$$
The spaces $p(2^\omega)$ and $P(2^\omega)$ are homeomorphic via the recursively defined isomorphism $f\mapsto \mu_f$ where $\mu_f\in P(2^\omega)$ is the measure uniquely determined by requiring that $\mu_f(N_s)=f(s)$ for all $s\in 2^{<\omega}$. We call the unique real $f\in p(2^\omega)$ such that $\mu=\mu_f$ the {\emph{code}} for $\mu$. The identification of $P(2^\omega)$ and $p(2^\omega)$ allow us to use the notions of effective descriptive set theory in the space $P(2^\omega)$. For instance, the set $P_c(2^\omega)$ of all non-atomic probability measures on $2^\omega$ is arithmetical because the set $p_c(2^\omega)=\{f\in p(2^\omega):\mu_f\;\hbox{is non-atomic}\}$ is easily seen to be arithmetical, as shown in \cite{VFAT}.

We will use the method of coding a real $z\in 2^\omega$ into a measure $\mu\in P_c(2^\omega)$ introduced in~\cite{VFAT}. For convenience we repeat the construction in minimal detail. Given $\mu\in P_c(2^\omega)$ and $s\in 2^{<\omega}$ we let $t(s,\mu)$ be the lexicographically least $t\in 2^{<\omega}$ such that $s\subseteq t$, $\mu(N_{t^\smallfrown 0})>0$ and $\mu (N_{t^\smallfrown 1})> 0$, if it exists and otherwise we let $t(s,\mu)=\emptyset$. Define recursively $t^\mu_n\in 2^{<\omega}$ by letting
$t^\mu_0=\emptyset$ and $t^\mu_{n+1}=t({t^\mu_n}^\smallfrown 0,\mu)$. Since $\mu$ is non-atomic, we have $\lh(t^\mu_{n+1})>\lh(t^\mu_n)$.
Let $t^\mu_\infty=\bigcup_{n=0}^\infty t^\mu_n$. For $f\in p_c(2^\omega)$ and $n\in\omega\cup\{\infty\}$ we will write $t^f_n$ for $t^{\mu_f}_n$. Clearly the sequence $(t^f_n:n\in\omega)$ is recursive in $f$.

Define the relation $R\subseteq p_c(2^\omega)\times 2^\omega$ as follows:
\begin{align*}
R(f,z)\iff  (\forall n\in\omega) &\big( z(n)=1 \longleftrightarrow (f({t_n^f}^\smallfrown 0)=\frac{2}{3}f(t^f_n) \wedge f(t_n^\smallfrown 1)=\frac{1}{3}f(t_n))\big)\\
\wedge &\big(z(n)=0 \leftrightarrow f({t_n^f}^\smallfrown 0)=\frac{1}{3}
f(t_n^f)\wedge f({t_n^f}^\smallfrown 1)=\frac{2}{3} f(t_n^f)\big).
\end{align*}

Whenever $(f,z)\in R$ we say that {\emph{$f$ codes $z$}}. Note that $\dom(R)=\{f\in p_c(2^\omega):(\exists z) R(f,z)\}$ is $\Pi^0_1$ and so the function $r:\dom(R)\to 2^\omega$, where $r(f)=z$ if and only if $(f,z)\in R$, is also $\Pi^0_1$. If $\nu$ is a measure such that $\nu=\mu_f$ for some code $f$, then let $r(\nu)=r(f)$. The key properties of this construction is contained in the following Lemma (see \cite[Coding Lemma]{VFAT}):

\begin{lemm}
There is a recursive function $G:p_c(2^\omega)\times 2^\omega\to p_c(2^\omega)$
such that $\mu_{G(f,z)}\approx \mu_f$ and $R(G(f,z),z)$ for all
$f\in p_c(2^\omega)$ and $z\in 2^\omega$.
\end{lemm}

The proofs of Theorems~\ref{TH1} and~\ref{TH2} use the following result, which we now prove.

\begin{propi}\label{PR1}
Let $a\in\mathbb R$ and suppose that there either is a Cohen real over $L[a]$ or there is a random real over $L[a]$. Then there is no $\Sigma^1_2(a)$ m.o. family.
\end{propi}

We first need a preparatory Lemma. In $2^\omega$, consider the equivalence $E_I$ defined by
$$
x E_I y\iff \sum_{n=0}^\infty\frac {|x(n)-y(n)|} {n+1}<\infty.
$$
We identify $2^\omega$ with $\mathbb Z_2^\omega$ and equip it with the Haar measure $\mu$.

\begin{lemm}\label{l.sum}
Let $A\subseteq 2^\omega$ be a Borel set such that $\mu(A)>0$. Then $E_I\leq_B E_I\upharpoonright A$, where $E_I\upharpoonright A$ is the restriction of $E_I$ to A.
\end{lemm}

{\it Notation:} The constant $0$ sequence of length $n\in\omega\cup\{\infty\}$ is denoted $0^n$. If $A\subseteq 2^\omega$ and $s\in 2^{<\omega}$ let
$$
A_{(s)}=\{x\in 2^\omega: s^\frown x\in A\},
$$
the \emph{localization} of $A$ at $s$.

\begin{proof}[Proof of Lemma \ref{l.sum}]
Without loss of generality assume that $A\subseteq 2^\omega$ is closed. We will define $q_n\in\omega$, $s_{n,i}, s_t\in 2^{<\omega}$ recursively for all $n\in\omega$, $i\in\{0,1\}$ and $t\in 2^{<\omega}$ satisfying
\begin{enumerate}
\item $q_0=0$ and $q_{n+1}=q_n+\lh(s_{n,0})$.
\item $s_{0,i}=\emptyset$ and $\lh(s_{n,i})=\lh(s_{n,1-i})>0$ when $n>0$.
\item $s_\emptyset=\emptyset$ and $s_{t^\frown i}={s_t}^\frown {s_{\lh(t)+1,i}}$ for all $t\in 2^{<\omega}$, $i\in\{0,1\}$.
\item $\frac 1 {n+1}\leq \sum_{k=0}^{\lh(s_{n+1,0})} \frac {|s_{n+1,0}(k)-s_{n+1,1}(k)|}{q_n+k+1}\leq \frac 2 {n+1}$.
\item $N_{s_t}\subseteq A$.
\item If $t\in 2^n$ then $\mu(A_{(s_t)})>1-2^{-n}$.
\end{enumerate}
Suppose this can be done. We claim that the map $2^\omega\to A: x\mapsto a_x$ defined by
$$a_x=\bigcup_{n\in\omega} s_{x\upharpoonright n}$$
is a Borel (in fact, continuous) reduction of $E_I$ to $E_I\upharpoonright A$. To see this, fix $x,y\in 2^\omega$ and note that by (4) we have that
$$
\sum_{n=0}^\infty \frac {|x(n)-y(n)|} {n+1}\leq\sum_{n=0}^\infty \sum_{k=0}^{\lh(s_{n+1,0})} \frac {|s_{n+1,x(i)}(k)-s_{n+1,y(i)}(k)|}{q_n+k+1}=\sum_{n=0}^\infty\frac {|a_x(n)-a_y(n)|}{n+1}\leq 2 \sum_{n=0}^\infty \frac {|x(n)-y(n)|} {n+1}
$$
so that $x E_I y$ if and only if $a_x E_I a_y$.

We now show that we can construct a scheme satisfying (1)--(6) above. Suppose $q_k$, $s_{k,i}$ and $s_t$ have been defined for all $k\leq n$ and $t\in 2^{\leq n}$. It is enough to define $s_{n+1,i}$ satisfying (4)--(6). Define
$$
f_{q_n}:2^\omega\to [0,\infty]: f_{q_n}(x)=\sum_{k=0}^\infty\frac {x(k)} {q_n+k+1}.
$$
It is clear that $f_{q_n}(N_{0^k})$ is dense in $[0,\infty]$ for all $k\in\omega$. Let
$$
A'=\{x\in A: \lim_{k\to\infty} \mu(A_{(x\upharpoonright k)})\to 1\},
$$
i.e, the set of points in $A$ of density 1. By the Lebesgue density theorem \cite[17.9]{KECHRIS} we have $\mu(A\setminus A')=0$. Let $A''=\bigcap_{t\in 2^n} A'_{(s_t)}$ and note that by (6) we have $\mu(A'')>0$. Thus the set of differences $A''-A''$ contains a neighborhood of $0^\infty$ by \cite[17.13]{KECHRIS}. It follows that there are $x_0,x_1\in A''$ such that
$$
\frac 1 {n+2}\leq \sum_{k=0}^\infty\frac {|x_0(k)-x_1(k)|}{q_n+k+1}\leq \frac 2 {n+2}.
$$
Since all points in $A'_{(s_t)}$ have density $1$ in $A'_{(s_t)}$ there is some $k_0\in\omega$ such that
$$
\mu(A'_{(s_t^\frown x_i\upharpoonright k_0)})>1-2^{-n-1}
$$
for all $t\in 2^n$. Defining $s_{n+1,i}=x_i\upharpoonright k_0$, it is then clear that (4)--(6) holds.
\end{proof}

\begin{proof}[Proof of Proposition \ref{PR1}]
As the proof easily relativizes, assume that $a=0$. We proceed exactly as in \cite[Proposition 4.2]{VFAT}. Suppose $A\subseteq P(2^\omega)$ is a $\Sigma^1_2$ m.o. family. Recall from \cite{kechsof01} and \cite[p. 1406]{VFAT} that there is a Borel function $2^\omega\to P(2^\omega):x\mapsto \mu^x$ such that
$$
x E_I y\Longrightarrow \mu^x\approx\mu^y
$$
and
$$
x\not\!\!{E}_I y\Longrightarrow \mu^x\perp\mu^y.
$$
Define as in \cite[Proposition 4.2]{VFAT} a relation $Q\subseteq 2^\omega\times P(2^\omega)^\omega$ by
$$
Q(x,(\nu_n))\iff (\forall n)(\nu_n\in A\wedge \nu_n\not\perp\mu^x)\wedge (\forall\mu)(\mu\not\perp\mu^x\longrightarrow (\exists n)\nu_n\not\perp\mu)
$$
and note that this is $\Sigma^1_2$ when $A$ is. Note that $Q(x,(\nu_n))$ precisely when $(\nu_n)$ enumerates the measures in $A$ not orthogonal to $\mu^x$ (this set is always countable, see \cite[Theorem 3.1]{kechsof01}.) Since $A$ is maximal, each section $Q_x$ is non-empty, and so we can uniformize $Q$ with a (total) function $f:2^\omega\to p(2^\omega)^\omega$ having a $\Delta^1_2$ graph. Note that assignment
$$
x\mapsto A(x)=\{f(x)_n:n\in\N\}
$$
is invariant on the $E_I$ classes.

If there is a Cohen real over $L$ it follows from \cite{jush89} that $f$ is Baire measurable. Since $E_I$ is a turbulent equivalence relation (in the sense of Hjorth, see e.g. \cite{kechsof01}) the map $x\mapsto A(x)$ must be constant on a comeagre set. But this contradicts that all $E_I$ classes are meagre.

If on the other hand there is a random real over $L$, then $f$ is Lebesgue measurable by \cite{jush89}. Let $F\subseteq 2^\omega$ be a closed set with positive measure on which $f$ is continuous, and let $g: 2^\omega\to F$ be a Borel reduction of $E_I$ to $E_I\upharpoonright F$. Note that $x\mapsto A(g(x))$ is then an $E_I$-invariant Borel assignment of countable subsets of $p(2^\omega)$, and so since $E_I$ is turbulent the function $f\circ g$ must be constant on a comeagre set. This again contradicts that all $E_I$ classes are meagre.
\end{proof}

\section{$\Delta^1_3$ w.o. of the reals, $\Pi^1_2$ m.o. family, no $\mathbf{\Sigma}^1_2$ m.o. families with $\mathfrak{b}=\mathfrak{c}=\omega_3$}

We proceed with the proof of Theorem~\ref{TH1}. We will use a modification of the model constructed in~\cite{VFSFLZ}. The preliminary stage $\PP_0=\PP^0*\PP^1*\PP^2$ of the iteration will coincide almost identically with the preliminary stage $\PP_0$ of~\cite{VFSFLZ} (see Step 0 through Step 2). For convenience of the reader we outline its construction. We work over the constructible universe $L$.

Recall that a transitive $ZF^-$ model is {\emph{suitable}} if $\omega_3^\M$ exists and $\omega_3^\M=\omega_3^{L^\M}$. If $\M$ is suitable then
also $\omega_1^\M=\omega_1^{L^\M}$ and $\omega_2^\M=\omega_2^{L^\M}$.

Fix a $\Diamond_{\omega_2}(\cof(\omega_1))$ sequence $\la G_\xi:\xi\in\omega_2\cap\cof(\omega_1)\ra$ which is $\Sigma_1$-definable over $L_{\omega_2}$. For $\alpha<\omega_3$, let $W_\alpha$ be the $L$-least subset of $\omega_2$ coding $\alpha$ and let $S_\alpha=\{\xi\in\omega_2\cap\cof(\omega_1):G_\xi=W_\alpha\cap\xi\neq\emptyset\}$. Then $\vec{S}=\la S_\alpha:1<\alpha<\omega_3\ra$
is a sequence of stationary subsets of $\omega_2\cap\cof(\omega_1)$, which are mutually almost disjoint.

For every $\alpha$ such that $\omega\leq\alpha<\omega_3$ shoot a club $C_\alpha$ disjoint from $S_\alpha$ via the poset $\PP^0_\alpha$, consisting of all closed subsets of $\omega_2$ which are disjoint from $S_\alpha$ with the extension relation being end-extension, and let $\PP^0=\prod_{\alpha<\omega_3}\PP^0_\alpha$ be the direct product of the $\PP^0_\alpha$'s with supports of size $\omega_1$, where for $\alpha\in\omega$, $\PP^0_\alpha$ is the trivial poset. Then $\PP^0$ is countably closed, $\omega_2$-distributive and $\omega_3$-c.c.

For every $\alpha$ such that $\omega\leq\alpha<\omega_3$ let $D_\alpha\subseteq\omega_3$ be a set coding the triple $\la C_\alpha, W_\alpha, W_\gamma\ra$ where $\gamma$ is the largest limit ordinal $\leq\alpha$. Let
$$E_\alpha=\{\M\cap\omega_2:\M\prec L_{\alpha+\omega_2+1}[D_\alpha],\omega_1\cup\{D_\alpha\}\subseteq\M\}.$$
Then $E_\alpha$ is a club on $\omega_2$. Choose $Z_\alpha\subseteq\omega_2$ such that $\mathit{Even}(Z_\alpha)=D_\alpha$, where $\mathit{Even}(Z_\alpha)=\{\beta :2\cdot\beta\in Z_\alpha\}$, and if $\beta<\omega_2$ is the $\omega_2^\M$ for some suitable model $\M$ such that $Z_\alpha\cap\beta\in\M$, then $\beta\in E_\alpha$. Then we have:

\smallskip
\begin{itemize}
\item[\ ] 
\begin{itemize}
\item[$(*)_\alpha$:]
   If  $\beta < \omega_2$, $\M$ is a suitable model  such that
   $\w_1\subset \M$, $\omega_2^\M = \beta$, and $Z_\alpha \cap \beta\in \M$, then
   $\M\vDash \psi(\w_2, Z_\alpha \cap \beta)$, where $\psi(\w_2,X)$ is the
   formula  ``$\mathit{Even}(X)$ codes a triple
   $\la\bar{C},\bar{W},\bar{\bar{W}}\ra$, where $\bar{W}$ and $\bar{\bar{W}}$ are
    the $L$-least codes of ordinals
   $\bar{\alpha},\bar{\bar{\alpha}} < \omega_3$ such that $\bar{\bar{\alpha}}$
   is the largest limit ordinal not exceeding $\bar{\alpha}$,
   and $\bar{C}$ is a club in $\omega_2$ disjoint
   from $S_{\bar{\alpha}}$''.
\end{itemize}
\end{itemize}

Similarly to  $\vec{S}$ define a sequence
$\vec{A}= \la A_\xi:\xi<\omega_2\ra$  of stationary subsets of
$\omega_1$ using the ``standard'' $\Diamond$-sequence. Code $Z_\alpha$ by a subset $X_\alpha$ of $\omega_1$
with the poset $\PP^1_\alpha$ consisting of all pairs
$\la s_0,s_1\ra\in[\omega_1]^{<\omega_1}\times [Z_\alpha]^{<\omega_1}$
where $\la t_0,t_1\ra\leq \la s_0,s_1\ra$ iff $s_0$ is an initial segment of
$t_0$, $s_1\subseteq t_1$ and $t_0\backslash s_0\cap
A_\xi=\emptyset$ for all $\xi\in s_1$. Then $X_\alpha$ satisfies the
following condition:
\smallskip

\begin{itemize}
\item[\ ]
\begin{itemize}
\item[$(**)_\alpha$:]

If $\w_1<\beta \leq \omega_2$ and $\M$ is a suitable model  such that
$\omega_2^\M = \beta$ and $\{X_\alpha\}\cup\w_1\subset \M$, then  $\M\vDash \phi(\w_1,\w_2, X_\alpha)$,
where $\phi(\w_1,\w_2,X)$ is the formula: `` Using the sequence $\vec{A}$,
$X$ almost disjointly codes a  subset $\bar{Z}$ of $\omega_2$, such that
$\mathit{Even}(\bar{Z})$ codes a triple $\la\bar{C},\bar{W},\bar{\bar{W}}\ra$, where
$\bar{W}$ and $\bar{\bar{W}}$ are the $L$-least codes of  ordinals
$\bar{\alpha},\bar{\bar{\alpha}} < \omega_3$ such that $\bar{\bar{\alpha}}$
is the largest limit ordinal not exceeding $\bar{\alpha}$,
and $\bar{C}$ is a club in $\omega_2$ disjoint from $S_{\bar{\alpha}}$''.
\end{itemize}
\end{itemize}

Let $\PP^1=\prod_{\alpha<\omega_3}\PP^1_\alpha$, where $\PP^1_\alpha$ is the trivial poset
for all $\alpha\in\omega$, with countable support. Then $\PP^1$ is countably closed and has the $\omega_2$-c.c.

Finally we force a localization of the $X_\alpha$'s. Fix $\phi$
as in $(**)_\alpha$ and let $\mathcal{L}(X,X')$ be the poset defined in~\cite[Definition 1]{VFSFLZ},
where $X, X'\subset\w_1$ are such that $\phi(\w_1,\w_2,X)$
and $\phi(\w_1,\w_2,X')$ hold in any suitable model $\M$ with $\w_1^\M=\w_1^L$
containing $X$ and $X'$, respectively. That is $\mathcal{L}(X,X')$ consists
of all functions $r:|r|\to 2$, where the domain $|r|$ of
$r$ is a countable limit ordinal such that:
\begin{enumerate}
{\item if $\gamma<|r|$ then $\gamma\in X$ iff $r(3\gamma)=1$}
{\item if $\gamma<|r|$ then $\gamma\in X'$ iff $r(3\gamma+1)=1$}
{\item if $\gamma\leq|r|$, $\M$ is a countable suitable model
containing $r\rest\gamma$ as an element and $\gamma=\omega_1^\M$,
then $\M\vDash\phi(\w_1,\w_2, X\cap\gamma)\wedge \phi(\w_1,\w_2, X'\cap\gamma)$.}
\end{enumerate}The extension relation is end-extension.
Then let $\PP^2_{\alpha+m}=\mathcal L(X_{\alpha+m}, X_\alpha)$ for every $\alpha\in\mathit{Lim}(\w_3)\backslash\{0\}$
and $m\in\omega$. Let $\PP^2_{\alpha+m}$ be the trivial poset for $\alpha=0$, $m\in\omega$ and let $$\PP^2=\prod_{\alpha\in\mathit{Lim}(\w_3)}\prod_{m\in\w}\PP^2_{\alpha+m}$$
with countable supports. Note that the  poset $\PP^2_{\alpha+m}$, where $\alpha>0$, produces a generic function
in  $^{\w_1}2$ (of $L^{\PP^0\ast\PP^1}$), which is the
characteristic function of a subset $Y_{\alpha+m}$ of $\w_1$ with the following property:
\smallskip

\begin{itemize}
\item[\ ] 
\begin{itemize}
\item[$(*\!*\!*)_\alpha$:]
For every $\beta < \omega_1$ and any suitable $\M$ such that $\omega_1^\M =
\beta$ and $Y_{\alpha+m} \cap \beta$ belongs to $\M$, we have
$\M\vDash\phi(\w_1,\w_2,X_{\alpha+m}\cap\beta)\wedge \phi(\w_1,\w_2,X_\alpha\cap\beta)$.
\end{itemize}
\end{itemize}

\begin{clm} $\PP_0:=\PP^0*\PP^1*\PP^2$ is $\omega$-distributive.
\end{clm}
\begin{proof}~\cite[Lemma 1]{VFSFLZ}.
\end{proof}

Let $\vec{B}=\la B_{\zeta,m}:\zeta<\omega_1,m\in\omega\ra$ be a nicely definable sequence of almost disjoint subsets of $\omega$. We will define a finite support
iteration $\la \PP_\alpha,\dot{\QQ}_\beta:\alpha\leq\omega_3,\beta<\omega_3\ra$ such that $\PP_0=\PP^0*\PP^1*\PP^2$, for every $\alpha<\omega_3$, $\dot{\QQ}_\alpha$ is a $\PP_\alpha$-name for a $\sigma$-centered poset, in $L^{\PP_{\omega_3}}$ there is a $\Delta^1_3$-definable wellorder of the
reals, a $\Pi^1_2$-definable maximal family of orthogonal measures and there are no $\Sigma^1_2$-definable maximal families of orthogonal measures. Along the iteration for every $\alpha<\omega_3$, we will define in $V^{\PP_\alpha}$ a set $O_\alpha$ of orthogonal measures and for $\alpha\in\mathit{Lim}(\alpha)$ a subset $A_\alpha$ of $[\alpha,\alpha+\omega)$. Every $\QQ_\alpha$ will add a generic real, whose $\PP_\alpha$-name will be denoted $\dot{u}_\alpha$ and similarly to the proof of~\cite[Lemma 2]{VFSFLZ} one can prove that $L[G_\alpha]\cap{^\omega\omega}=L[\la\dot{u}^{G_\alpha}_\xi:\xi<\alpha\ra]\cap{^\omega\omega}$ for every $\PP_\alpha$-generic filter $G_\alpha$.  This gives a canonical wellorder of the reals in $L[G_\alpha]$ which depends only on the sequence $\la \dot{u}_\xi:\xi<\alpha\ra$, whose $\PP_\alpha$-name will be denoted by $\dot{<}_\alpha$. We can additionally arrange that for $\alpha<\beta$, $<_\alpha$ is an initial segment of
$<_\beta$, where $<_\alpha=\dot{<}^{G_\alpha}_\alpha$ and $<_\beta=\dot{<}_\beta^{G_\beta}$. Then if $G$ is a $\PP_{\omega_3}$-generic filter over $L$, then $<^G=\bigcup\{\dot{<}^G_\alpha:\alpha<\omega_3\}$ will be the desired wellorder of the reals and
$O=\bigcup_{\alpha<\omega_3} O_\alpha$ will be the $\Pi^1_2$-definable maximal family of orthogonal measures.

We proceed with the recursive definition of $\PP_{\omega_3}$. For every $\nu\in[\omega_2,\omega_3)$ let
$i_\nu:\nu\cup\{\la\xi,\eta\ra:\xi<\eta<\nu\}\to\mathit{Lim}(\omega_3)$ be a fixed bijection. If $G_\alpha$ is a
$\PP_\alpha$-generic filter over $L$, $<_\alpha=\dot{<}_\alpha^{G_\alpha}$ and $x,y$  are reals in $L[ G_\alpha]$
such that $x<_\alpha y$, let $x*y=\{2n:n\in x\}\cup\{2n+1:n\in y\}$ and $\Delta(x*y)=\{2n+2:n\in x*y\}\cup\{2n+1:n\notin x*y\}$.
Suppose $\PP_\alpha$ has been defined and fix a $\PP_\alpha$-generic filter $G_\alpha$.

If $\alpha=\omega_2\cdot\alpha^\prime+\xi$, where $\alpha^\prime>0$, $\xi\in\mathit{Lim}(\omega_2)$, let $\nu=o.t.(\dot{<}^{G_\alpha}_{\omega_2\cdot\alpha^\prime})$ and let $i=i_\nu$.

{\emph{Case 1}}. If $i^{-1}(\xi)=\la\xi_0,\xi_1\ra$ for some $\xi_0<\xi_1<\nu$, let $x_{\xi_0}$ and $x_{\xi_1}$ be the $\xi_0$-th and
$\xi_1$-th reals in $L[G_{\omega_2\cdot\alpha^\prime}]$ according to the wellorder $\dot{<}^{G_\alpha}_{\omega_2\cdot\alpha^\prime}$. In $L^{\PP_\alpha}$
let $$\QQ_\alpha=\{\la s_0,s_1\ra:s_0\in[\omega]^{<\omega}, s_1\in [\bigcup_{m\in\Delta(x_{\xi_0}*x_{\xi_1})} Y_{\alpha+m}\times{\{m\}}]^{<\omega}\},$$
where $\la t_0, t_1\ra\leq \la s_0,s_1\ra$ if and only if $s_1\subseteq t_1$, $s_0$ is an initial segment of $t_0$ and $(t_0\backslash s_0)\cap B_{\zeta,m}=\emptyset$ for all $\la \zeta, m\ra\in s_1$. Let $u_\alpha$ be the generic real added by $\QQ_\alpha$, $A_\alpha=\alpha+\omega\backslash
\Delta(x_{\xi_0}*x_{\xi_1})$ and $O_\alpha=\emptyset$.

{\emph{Case 2}}. Suppose $i^{-1}(\xi)=\zeta\in\nu$. If the $\zeta$-th real according to the wellorder $\dot{<}^{G_\alpha}_{\omega_2\cdot\alpha^\prime}$
is not the code of a measure orthogonal to $O^\prime_\alpha=\bigcup_{\gamma<\alpha}O_\gamma$, let $\QQ_\alpha$ be the trivial poset, $A_\alpha=\emptyset$, $O_\alpha=\emptyset$. Otherwise, i.e.
in case $x_\zeta$ is a code for a measure orthogonal to $O^\prime_\alpha$, let
$$\QQ_\alpha=\{\la s_0,s_1\ra:s_0\in[\omega]^{<\omega}, s_1\in [\bigcup_{m\in\Delta(x_\zeta)} Y_{\alpha+m}\times{\{m\}}]^{<\omega}\},$$
where $\la t_0, t_1\ra\leq \la s_0,s_1\ra$ if and only if $s_1\subseteq t_1$, $s_0$ is an initial segment of $t_0$ and $(t_0\backslash s_0)\cap B_{\zeta,m}=\emptyset$ for all $\la \zeta, m\ra\in s_1$. Let $u_\alpha$ be the generic real added by $\QQ_\alpha$. In $L^{\PP_{\alpha+1}}= L^{\PP_\alpha*\QQ_\alpha}$ let $g_\alpha=G(x_\zeta,u_\alpha)$ be the code of a measure equivalent to $\mu_{x_\zeta}$ which codes
$u_\alpha$ (see~\cite[Lemma 3.5]{VFAT}) and let $O_\alpha=\{\mu_{g_\alpha}\}$. Let $A_\alpha=\alpha+\omega\backslash\Delta(u_\alpha)$.

If $\alpha$ is not of the above form, i.e. $\alpha$ is a successor or $\alpha\in\omega_2$, let $\QQ_\alpha$ be the following poset for adding a dominating real:
$$\QQ_\alpha=\{\la s_0,s_1\ra:s_0\in \omega^{<\omega}, s_1\in [\hbox{o.t.}(\dot{<}_\alpha^{G_\alpha})]^{<\omega}\},$$
where $\la t_0, t_1\ra\leq \la s_0,s_1\ra$ if and only if $s_0$ is an initial segment of  $t_0$, $s_1\subseteq t_1$, and $t_0(n)>x_\xi(n)$ for all
$n\in\hbox{dom}(t_0)\backslash\hbox{dom}(s_0)$ and $\xi\in s_1$, where $x_\xi$ is the $\xi$-th real in $L[G_\alpha]\cap \omega^\omega$ according to the wellorder $\dot{<}^{G_\alpha}_\alpha$. Let $A_\alpha=\emptyset$, $O_\alpha=\emptyset$.

With this the definition of $\PP_{\omega_3}$ is complete. Let $O=\bigcup_{\alpha<\omega_3} O_\alpha$. In $L^{\PP_{\omega_3}}$ we have: {\emph{$\nu$ is a measure in the set $O$ if and only if
for every countable suitable model $\M$ such that $\nu\in\M$, there is $\bar{\alpha}<\omega_3^\M$ such that $S_{\bar{\alpha}+m}$ is nonstationary in
$(L[r(\nu)])^\M$ for every $m\in \Delta(r(\nu))$.}}
Therefore $O$ has indeed a $\Pi^1_2$ definition. Furthermore $O$ is maximal in $P_c(2^\omega)$. Indeed, suppose in $L^{\PP_{\omega_3}}$ there is a code $x$ for a measure orthogonal to every measure in the family $O$. Choose $\alpha$ minimal such that $\alpha=\omega_2\cdot\alpha^\prime+\xi$ for some $\alpha^\prime>0$ and $\xi\in\mathit{Lim}(\omega_2)$ and $x\in L[G_{\omega_2\cdot\alpha^\prime}]$. Let $\nu=o.t.(\dot{<}^{G_\alpha}_{\omega_2\cdot\alpha^\prime})$ and let
$i=i_\nu$. Then $x=x_\zeta$ is the $\zeta$-th real according to the wellorder $\dot{<}^{G_\alpha}_{\omega_2\cdot\alpha^\prime}$, where $\zeta\in\nu$ and so for some $\xi\in\mathit{Lim}(\omega_2)$, $i^{-1}(\xi)=\zeta$. But then $x_\zeta=x$ is the code of a measure orthogonal to $O_\alpha$ and so by construction $O_{\alpha+1}$ contains a measure equivalent to $\mu_x$, which is a contradiction. To obtain a $\Pi^1_2$-definable m.o. family in $L{\PP_{\omega_3}}$ consider the union of $O$ with the set of all point measures. Just as in~\cite{VFSFLZ} one can show that $<$ is indeed a $\Delta^1_3$-definable wellorder of the reals.

Since $\PP_{\omega_3}$ is a finite support iteration, we have added Cohen reals along the iteration cofinally often. Thus for every real $a$ in $L^{\PP_{\omega_3}}$ there is a Cohen real over $L[a]$ and so by Proposition~\ref{PR1} in $L^{\PP_{\omega_3}}$ there are no
$\mathbf{\Sigma}^1_2$ m.o. families. Also note that since cofinally often we have added dominating reals, $L^{\PP_{\omega_3}}\vDash\mathfrak{b}=\omega_3$.

\section{$\Delta^1_3$ w.o. of the reals, a $\Pi^1_2$ m.o. family, no $\mathbf{\Sigma}^1_2$ m.o. families with $\mathfrak{c}=\omega_2$}

In this section we establish the proof of Theorem~\ref{TH2}. The model is obtained as a slight modification of the iteration construction developed in~\cite{VFSF}. We restate the definitions of the posets used in this construction. For a more detailed account of their properties see~\cite{VFSF}.
We work over the constructible universe $L$.

If $S\subseteq\omega_1$ is a stationary, co-stationary set, then by $Q(S)$ denote the poset of all countable closed subsets of $\omega_1\backslash S$ with the extension relation being end-extension. Recall that $Q(S)$ is $\omega_1\backslash S$-proper, $\omega$-distributive and adds a club disjoint from $S$
(see~\cite{VFSF},~\cite{MGOLD}). For the proof of Theorem~\ref{TH2} we use the form of localization defined in~\cite[Definition 1]{VFSF}. That is, if $X\subseteq\omega_1$ and $\phi(\omega_1,X)$ is a $\Sigma_1$-sentence with parameters $\omega_1,X$ which is true in all suitable models containing
$\omega_1$ and $X$ as elements, then $\LL(\phi)$ be the poset of all functions $r:|r|\to 2$, where the domain $|r|$ of $r$ is a countable limit ordinal, such that
\begin{enumerate}
{\item if $\gamma<|r|$ then $\gamma\in X\;\hbox{iff}\;r(2\gamma)=1$}
{\item if $\gamma\leq |r|$, $\M$ is a countable, suitable model containing $r\rest\gamma$ as an element and $\gamma=\omega_1^\M$,  then $\phi(\gamma, X\cap\gamma)$ holds in $\M$.}
\end{enumerate}
The extension relation is end-extension. Recall that $\LL(\phi)$ has a countably closed dense subset (see~\cite[Remark 2]{VFSF}) and that if $G$ is $\LL(\phi)$-generic and $\M$ is a countable suitable model containing $(\bigcup G)\rest \gamma$ as an element, where $\gamma=\omega_1^\M$, then
$\M\vDash\phi(\gamma,X\cap\gamma)$ (see~\cite[Lemma 2]{VFSF}).

We will use also the coding with perfect trees defined in~\cite[Definition 2]{VFSF}.  Let $Y\subseteq \omega_1$ be generic over $L$ such that in $L[Y]$ cofinalities have not been changed and let $\bar{\mu}=\{\mu_i\}_{i\in\omega_1}$ be a sequence of $L$-countable ordinals such that $\mu_i$ is the least $\mu>\sup_{j<i}\mu_j$, $L_\mu[Y\cap i]\vDash ZF^-$ and $L_\mu\vDash\omega\;\hbox{is the largest cardinal}$. Say that a real $R$ {\emph{codes}} $Y$ below $i$ if for all $j<i$, $j\in Y$ if and only if $L_{\mu_j}[Y\cap j, R]\vDash ZF^-$. For $T\subseteq 2^{<\omega}$ a perfect tree, let $|T|$ be the least $i$ such that $T\in L_{\mu_i}[Y\cap i]$. Then $\Ca(Y)$ is the poset of all perfect trees $T$ such that $R$ codes $Y$ below $|T|$, whenever $R$ is a branch through $T$, where for  $T_0,T_1$ conditions in $\Ca(Y)$, $T_0\leq T_1$ if and only if $T_0$ is a subtree of $T_1$. Recall also that $\Ca(Y)$
is proper and $^\omega\omega$-bounding (see~\cite[Lemmas 7,8]{VFSF}).

Fix a bookkeeping function $F:\omega_2\to L_{\omega_2}$ and a sequence $\vec{S}=(S_\beta:\beta<\omega_2)$ of almost disjoint stationary subsets of $\omega_1$, defined
as in~\cite[Lemma 14]{VFSF}. Thus $F$ and $\vec{S}$ are $\Sigma_1$-definable over $L_{\omega_2}$ with parameter $\omega_1$, $F^{-1}(a)$  is unbounded in $\omega_2$ for every $a\in L_{\omega_2}$ and whenever $\M,\N$ are suitable models such that $\omega_1^\M=\omega_1^\N$ then $F^\M, \vec{S}^\M$ agree
with $F^\N$, $\vec{S}^\N$ on $\omega_2^\M\cap\omega_2^\N$. Also if $\M$ is suitable and $\omega_1^\M=\omega_1$ then $F^\M, \bar{S}^\M$ equal the restrictions of $F$, $\vec{S}$ to the $\omega_2$ of $\M$. Fix also a stationary subset $S$ of $\omega_1$ which is almost disjoint from every element of $\vec{S}$.

Recursively we will define
a countable support iteration $\la \PP_\alpha,\dot{\QQ}_\beta:\alpha\leq\omega_2,\beta<\omega_2\ra$ and a sequence $\la O_\alpha:\alpha\in\omega_2\ra$, such that in $L^{\PP_{\omega_2}}$ there is a $\Delta^1_3$-definable wellorder of the reals and $O=\bigcup_{\alpha<\omega_2}O_\alpha$ is a maximal
family of orthogonal measures. Define the wellorder $<_\alpha$ in $L[G_\alpha]$ where $G_\alpha$ is $\PP_\alpha$-generic just as in~\cite{VFSF}. We can assume that all names for reals are nice and that for $\alpha<\beta<\omega_2$, all $\PP_\alpha$-names for reals precede in the canonical wellorder
$<_L$ of $L$ all $\PP_\beta$-names for reals, which are not $\PP_\alpha$-names. For each $\alpha<\omega_2$, define a wellorder $<_\alpha$ on the reals of $L[G_\alpha]$, where $G_\alpha$ is a $\PP_\alpha$-generic as follows. If $x$ is a real in $L[G_\alpha]$ let $\sigma^\alpha_x$ be the $<_L$-least $\PP_\gamma$-name for $x$, where $\gamma\leq\alpha$ is least
so that $x$ has a $\PP_\gamma$-name. For $x,y$ reals in $L[G_\alpha]$ define $x<_\alpha y$ if and only if $\sigma^\alpha_x<_L \sigma^\alpha_y$.
Note that whenever $\alpha<\beta$, then $<_\alpha$ is an initial segment of $<_\beta$.

We proceed with the definition of the poset. Let $\PP_0$ be the trivial poset. Suppose $\PP_\alpha$ and $\la O_\gamma:\gamma<\alpha\ra$ have been defined. Let $\dot{\QQ}_\alpha=\dot{\QQ}^0_\alpha*\dot{\QQ}^1_\alpha$ be a $\PP_\alpha$-name for a poset where $\dot{\QQ}_\alpha^0$ is a $\PP_\alpha$-name for the random real forcing and $\dot{\QQ}^1_\alpha$ is defined as follows:

{\emph{Case 1.}} If $F(\alpha)=\{\sigma^\alpha_x,\sigma^\alpha_y\}$ for some pair of reals $x,y$ in $L[G_\alpha]$, then define $\QQ_\alpha$ as in~\cite{VFSF}. That is $\QQ_\alpha$ is a three stage iteration $\KK^0_\alpha*\dot{\KK}_\alpha^1*\dot{\KK}_\alpha^2$ where:

\bigskip
\noindent
$(1)$ In $V^{\PP_\alpha*\dot{\QQ}_\alpha^0}$, $\KK^0_\alpha$ is the direct limit $\la \PP^0_{\alpha,n},\dot{\KK}^0_{\alpha,n}:n\in\omega\ra$, where  $\dot{\KK}^0_{\alpha,n}$ is a $\PP^0_{\alpha,n}$-name for $Q(S_{\alpha+2n})$ for $n\in x_\alpha*y_\alpha$, and
$\dot{\KK}^0_{\alpha,n}$ is a $\PP^0_{\alpha,n}$-name for $Q(S_{\alpha+2n+1})$ for $n\not\in x_\alpha*y_\alpha$.

\bigskip
\noindent
$(2)$ Let $G_\alpha^0$ be a $\PP_\alpha*\dot{\QQ}^0_\alpha$-generic filter and let $H_\alpha$ be a $\KK^0_\alpha$-generic over $L[G_\alpha^0]$. In $L[G_\alpha^0*H_\alpha]$ let $X_\alpha$ be a subset
of $\omega_1$ coding $\alpha$, coding the pair $(x_\alpha,y_\alpha)$, coding a level of $L$ in which $\alpha$ has size at most $\omega_1$ and coding the generic $G_\alpha^0*H_\alpha$, which we can regard as a subset of an element of $L_{\omega_2}$. Let $\KK_\alpha^1=\LL(\phi_\alpha)$ where $\phi_\alpha=\phi_\alpha(\omega_1,X)$ is the $\Sigma_1$-sentence which holds if and only if $X$ codes an ordinal $\bar{\alpha}<\omega_2$ and a pair $(x,y)$ such that  $S_{\bar{\alpha}+2n}$ is nonstationary for $n\in x*y$ and $S_{\bar{\alpha}+2n+1}$ is nonstationary for $n\not\in x*y$. Let $\dot{X}_\alpha$ be a $\PP^0_\alpha*\dot{\QQ}_\alpha^0*\dot{\KK}_\alpha^0$-name for $X_\alpha$ and let $\dot{\KK}_\alpha^1$ be a $\PP_\alpha^0*\dot{\QQ}_\alpha^0*\dot{\KK}_\alpha^0$-name for $\KK_\alpha^1$.

\bigskip
\noindent
$(3)$ Let $Y_\alpha$ be $\KK^1_\alpha$-generic over $L[G_\alpha^0*H_\alpha]$. Note that the even part of $Y_\alpha$-codes $X_\alpha$ and so codes the generic $G_\alpha^0*H_\alpha$. Then in $L[Y_\alpha]=L[G_\alpha^0*H_\alpha*Y_\alpha]$, let $\KK^2_\alpha=\Ca(Y_\alpha)$. Finally, let
$\dot{\KK}^2_\alpha$ be a $\PP_\alpha*\dot{\QQ}^0_\alpha*\dot{\KK}^0_\alpha*\dot{\KK}^1_\alpha$-name for $\KK^2_\alpha$.

{\emph{Case 2.}} If $F(\alpha)=\{\sigma^\alpha_x\}$ where $x$ is a code for a measure orthogonal to $\bigcup_{\gamma<\alpha} O_\gamma$, then let $\dot{\QQ}^1_\alpha$
be a $\PP_\alpha*\dot{\QQ}^1_\alpha$-name for $\KK^0_\alpha*\dot{\KK}^1_\alpha*\dot{\KK}^2_\alpha$ where in $L^{\PP_\alpha*\dot{\QQ}_\alpha}$, $\KK^0_\alpha$ is the direct limit $\la \PP^0_{\alpha,n},\dot{\QQ}^0_{\alpha,n}:n\in\omega\ra$
where $\dot{\QQ}^0_{\alpha,n}$ is a $\PP^0_{\alpha,n}$-name for $Q(S_{\alpha+2n})$ for every $n\in x$ and
a $\PP^0_{\alpha,n}$-name for $Q(S_{\alpha+2n+1})$ for every $n\notin x$. Define $\KK^1_\alpha$ and $\KK^2_\alpha$ just as in {\emph{Case 1}}. In $L^{\PP_\alpha*\QQ_\alpha}$ let $g=G(x,R_\alpha)$ be a code for a measure which is equivalent to $\mu_x$ and codes the real $R_\alpha$. Let $O_\alpha=\{\mu_g\}$.

In any other case, let $\QQ_\alpha$ be a $\PP_\alpha$-name for the trivial poset, $O_\alpha=\emptyset$. With this the definition of $\PP_{\omega_2}$ and the family $O=\bigcup_{\gamma<\omega_2}O_\alpha$ is complete.

\begin{clm} $O=\bigcup_{\gamma<\omega_2} O_\gamma$ is a maximal family of orthogonal measures in $P_c(2^\omega)$.
\end{clm}
\begin{proof}
It is clear that $O$ is a family of orthogonal measures. It remains to verify its maximality. Suppose the contrary and
let $f$ be a code for a measure in $L[G]$ where $G$ is $\PP_{\omega_3}$-generic over $L$, which is orthogonal to all measures in $O$. Fix $\alpha$ minimal such that $f$ is in $L[G_\alpha]$
and let $\sigma$ be the $<_L$-least name for $f$. Since $F^{-1}(\sigma)$ is unbounded, there is $\beta\geq\alpha$ such that $F(\beta)=\{\sigma\}$.
Therefore $\QQ_\beta$ is nontrivial and $O_\beta=\{\mu_g\}$ for some measure $\mu_g$ which is equivalent to $\mu_f$, which is a contradiction.
\end{proof}

Clearly,{\emph{ $\mu\in O$ if and only if for every countable suitable model $\M$ such that $\mu\in \M$ there is $\alpha<\omega_2^\M$
such that $S_{\alpha+m}$ is nonstationary in $L[r(\mu)]^\M$ for every $m\in\Delta(r(\mu))$.}} Thus our family $O$ has indeed a $\Pi^1_2$ definition.
Just as in the proof of Theorem~\ref{TH1}, to obtain a $\Pi^1_2$-definable m.o. family in $L^{\PP_{\omega_3}}$ consider the union of $O$ with the set of all point measures.

Since for every real $a\in L^{\PP_{\omega_3}}$ there is a random real over $L$, by Proposition~\ref{PR1} in $L^{\PP_{\omega_3}}$ there are no $\mathbf{\Sigma}^1_2$ m.o. families. The bounding number $\mathfrak{b}$ remains $\omega_1$ in $L^{\PP_{\omega_3}}$, since the countable support iteration of $S$-proper $^\omega\omega$-bounding posets is $^\omega\omega$-bounding (see~\cite[Lemma 18]{VFSF} or~\cite{MGOLD}).\hfill\qedsymbol

\begin{remark}

In \cite{VFAT} the following question was raised:

\begin{question}
If there is a $\Pi^1_1$ m.o.~family, are all reals constructible?
\end{question}

This is to our knowledge still unsolved. T\"ornquist has recently shown that the existence of a $\Sigma^1_2$ m.o.~family implies the existence of a $\Pi^1_1$ m.o.~family, and that the existence of $\Sigma^1_2$ mad family implies the existence of a $\Pi^1_1$ mad family.

\end{remark}

\end{document}